\documentclass[oneside,11pt]{amsart}
\usepackage{amsmath, amsfonts,amsthm,times,graphics}

\usepackage[active]{srcltx}
 \makeatletter
\renewcommand*\subjclass[2][2000]{%
  \def\@subjclass{#2}%
  \@ifundefined{subjclassname@#1}{%
    \ClassWarning{\@classname}{Unknown edition (#1) of Mathematics
      Subject Classification; using '1991'.}%
  }{%
    \@xp\let\@xp\subjclassname\csname subjclassname@#1\endcsname
  }%
}
 \makeatother
\newtheorem{theorem}{Theorem}[section]
\newtheorem{lemma}[theorem]{Lemma}
\newtheorem{corollary}[theorem]{Corollary}
\newtheorem{proposition}[theorem]{Proposition}
\theoremstyle{definition}
\newtheorem{definition}[theorem]{Definition}

\newtheorem{remark}[theorem]{Remark}
\numberwithin{equation}{section}

\begin{document}
\title[On quasiconformal mappings and elliptic PDE in the plane]{On quasiconformal selfmappings of the unit disk and elliptic PDE in the plane}



\author{David Kalaj}
\address{ Faculty of Natural Sciences and
Mathematics, University of Montenegro, Dzordza Vasingtona b.b. 81000
Podgorica, Montenegro} \email{davidk@ac.me} \subjclass {Primary
30C55, Secondary 31C05}
\begin{abstract}
We prove the following theorem: if $w$ is a quasiconformal mapping
of the unit disk onto itself satisfying elliptic partial
differential inequality $|L[w]|\le \mathcal{B}|\nabla w|^2+\Gamma$,
then $w$ is Lipschitz continuous. This { result} extends some recent
results, where instead of an elliptic differential operator is {
only} considered { the} Laplace operator.
\end{abstract} \maketitle

\section{Introduction and notation}
\subsection{Quasiconformal mappings}
Let $A=\begin{pmatrix}
  a^{11} & a^{12} \\
  a^{21} & a^{22}
\end{pmatrix}.$ We will consider the
matrix norm:  $$|A|=\max\{|Az|: z\in \mathbf R^2, |z|=1\}\
$$ and the matrix function $$l(A)=\min\{|Az|: z\in \mathbf R^2, |z|=1\}.$$  Let $D$ and $\Omega$ be subdomains of the complex plane
$\mathbf C$, and $w=u+iv:D\to \Omega$ be a function that has both
partial derivatives at a point $z\in D$. By $\nabla w(z)$ we denote
the matrix $\begin{pmatrix}
  u_{x} & u_{y} \\
  v_{x} & v_{y}
\end{pmatrix}.$ For the matrix $\nabla
w$ we have \begin{equation}\label{opernorm}|\nabla w|=|\partial
w|+|\bar \partial w|\end{equation} and
\begin{equation}\label{ll}l(\nabla w)=||\partial
w|-|\bar \partial w||,\end{equation} where
$$\partial w=w_z := \frac{1}{2}\left(w_x+\frac{1}{i}w_y\right)\text{ and } \bar \partial w=w_{\bar z} := \frac{1}{2}\left(w_x-\frac{1}{i}w_y\right).$$
We say that a function $u:D\to \mathbf R$ is ACL (absolutely
continuous on lines) in  the region $D$, if for every closed
rectangle $R\subset D$ with sides parallel to the $x$ and $y$-axes,
$u$ is absolutely continuous on a.e. horizontal and a.e. vertical
line in $ R$. Such a function has of course, partial derivatives
$u_x$, $u_y$ a.e. in $D$.

A sense-preserving homeomorphism  $w\colon D\to \Omega, $ where $D$
and $\Omega$ are subdomains of the complex plane $\mathbf C,$ is
said to be $K$-quasiconformal ($K$-q.c), { with} $K\ge 1$, if $w$ is
ACL in $D$ in the sense that the real and imaginary part are ACL in
D, and
\begin{equation}\label{defqc} |\nabla w|\le K
l(\nabla w)\ \ \ \text{a.e. on $D$},\end{equation} (cf. \cite{Ahl},
pp. 23--24). Notice that the condition (\ref{defqc}) can be written
as
$$|w_{\bar z}|\le k|w_z|\quad \text{a.e. on $D$ where
$k=\frac{K-1}{K+1}$ i.e. $K=\frac{1+k}{1-k}$ }.$$ If in the previous
definition { we} replace the condition "$w$ is  a sense-preserving
homeomorphism" by the condition "$w$ is continuous", { then} we
obtain the definition of a quasiregular mapping.
\subsection{Elliptic operator}
Let $A(z)= \{a^{ij}(z)\}_{i,j=1}^2$ be a symmetric  matrix function
defined in a domain $D\subset \mathbf C$ $(a^{ij}= a^{ji})$. Assume
that
\begin{equation}\label{(2.11)}\Lambda^{-1} \le \left<A(z)h,h\right>\le \Lambda \ \ \ \text{
for} \ \ \ |h| = 1,\end{equation} where $\Lambda$ is a constant $\ge
1$ or { written} in coordinates
\begin{equation}\label{(2.1)}\Lambda^{-1} \le \sum_{i,j=1}^2 a^{ij}(z)h_i h_j \le \Lambda \text{
for} \sum_{i=1}^2 h_i^2 = 1.\end{equation} In addition for a certain
$\mathfrak L\ge 0$, we { suppose} that
\begin{equation}\label{con2} |A(z)-A(\zeta)|\le \mathfrak L|\zeta-z|   \ \
\text{ for any \ \  $z,\zeta\in D$.}\end{equation} For
\begin{equation}\label{elli}L[u] := \sum_{i,j=1}^2 a^{ij}(z)
D_{ij}u(z),\end{equation} subjected to conditions \eqref{(2.1)} and
\eqref{con2} we consider the following differential inequality
\begin{equation}\label{(0)}|L[u]|\le \mathcal{B}|\nabla u|^2+\Gamma,\end{equation} with given $\mathcal{B}$, $\Gamma\ge 0$, or, by using Einstein
convention
\begin{equation}\label{short}|a^{ij}(z)D_{ij}u| \le \mathcal{B}|\nabla
u|^2+\Gamma,\end{equation} and call it {\it elliptic partial
differential inequality.} Observe that, if $A$ is the identity
matrix, then  $L$ is the Laplace operator $\Delta$. A $C^2$
solutions $u: D\to \mathbf R(\mathbf C)$ of the equation $\Delta u =
0$ is called a harmonic function ({ mapping}) and the corresponding
inequality \eqref{elli} is called {\it Poisson differential
inequality}. The class of harmonic quasiconformal mappings (HQC) has
been one of recent main topics of investigation of some authors. See
the subsection below. For the connection between quasiconformal
mappings and PDE we refer to the book \cite{ar1}. See also
\cite[Chapter~12]{gt}, \cite{robfin}, \cite{nir} and \cite{ls}.
\subsection{Background and statement of the main result}
Let $\gamma$ be a Jordan curve. By the Riemann mapping theorem,
there exists a Riemann conformal mapping of the unit disk onto a
Jordan domain $\Omega=\mathrm{int}\ \gamma$. By Caratheodory's
theorem, it has a continuous extension to the boundary. Moreover, if
$\gamma\in C^{1,\alpha}$, $0< \alpha <1$, then the Riemann conformal
mapping has $C^{1,\alpha}$ extension to the boundary (this result is
known as Kellogg's theorem). We refer to \cite{G} for the proof of
the previous result and \cite{PW, POW, Ko, Lw} for related results.
In particular a conformal mapping $w$ of the unit disk onto a Jordan
domain $\Omega$ with $C^{1,\alpha}$ boundary is Lipschitz
continuous, i.e. it satisfies the inequality $|w(z)-w(z')|\le
C|z-z'|$, $z,z'\in\mathbf U:=\{z\in\mathbf C: |z|<1\}$.

On the other hand $K$ quasiconformal mappings between smooth domains
are H\"older continuous and the best H\"older constant is $1/K$. So
they are not in general Lipschitz mappings, except if $K=1$. In this
paper we are concerned with an additional condition of a
quasiconformal mapping in order to guaranty its global Lipschitz
character.

One of "additional condition" is to assume harmonicity of the
mapping. This condition is natural since conformal mappings are
quasiconformal and harmonic. Hence,  quasiconformal harmonic
(shortly HQC) mappings are natural generalization of conformal
mappings.
O. Martio \cite{Om} was the first who considered harmonic
quasiconformal mappings on the complex plane.

Recently, there has been a number of authors who are working on {
this} topic.
We list below some of related results:

1) If $w$ is harmonic quasiconformal mapping of the unit disk onto
itself, then $w$ is Lipschitz (Pavlovic theorem proved in
\cite{MP}). See also some refinements of Partyka and Sakan
\cite{pk}.

2) If $w$  is a harmonic quasiconformal mapping between two
$C^{1,\alpha}$ Jordan domains, then $w$ is Lipschiz (the result of
the author  proved in \cite{mathz}).

3) If $w$ is a quasiconformal mapping between two $C^{2,\alpha}$
Jordan domains satisfying the partial differential inequality
$|\Delta w|\le C |f_z f_{\bar z}|,$ then $w$ is Lipschitz (the
author $\&$ Mateljevi\'c result proved in \cite{km}).

4) If $w$ is a quasiconformal mapping of the unit disk onto itself
satisfying the PDE  $\Delta w= g$ then this mapping is Lipschiz (the
author $\&$ Pavlovi\'c result proved in \cite{trans}).

5) If $w$ is a quasiconformal mapping between two $C^{2,\alpha}$
Jordan domains satisfying the partial differential inequality
$|\Delta w|\le \mathcal B |\nabla w|^2 + \Gamma$, then $w$ is
Lipschiz (the author $\&$ Mateljevi\'c result proved in
\cite{pota}).

Notice that the proofs of 3)--5) depend on a Heinz theorem, see
\cite{EH}.

Concerning the bi-Lipschitz character of the class HQC we refer to
the papers \cite{pisa}, \cite{kalann}, \cite{kojic}, \cite{MMM} and
\cite{xa}. See also \cite{kalmat} and \cite{mv2} for some results
concerning higher dimensional case.

For related result about quasiconformal harmonic mappings with
respect to the hyperbolic metric we refer to the paper of Wan
\cite{wan} and of Markovi\'c \cite{markovic}.

{ More} recently, Iwaniec, Kovalev and Onninen in \cite{nit} have
shown that the class of quasiconformal harmonic mappings is also
interesting concerning the modulus of annuli in complex plane.

In this paper we study Lipschitz continuity of the class of $K$-q.c.
self-mappings of the unit disk satisfying elliptic differential
inequality $|Lw| \le \mathcal B|\nabla w|^2+\Gamma$. This class
contains conformal mappings and quasiconformal harmonic mappings.

The main result of this paper is the following theorem which is an
extension of results 1)-- 5) mentioned above.
\begin{theorem}\label{mama}
If $a\in\mathbf U$, and $w:\mathbf U\to \mathbf U$, $w(a)=0$ and
$w(\mathbf U)=\mathbf{U}$ is a $K$ q.c. solution of the elliptic
partial differential inequality
\begin{equation}\label{ii}|L[w]|\le \mathcal B|\nabla
w|^2+\Gamma,\end{equation} then $\nabla w$ is bounded by a constant
$C(K,\mathcal B,\Gamma,\Lambda, \mathfrak L, a)$ and $w$ is {
Lipschitz} continuous.
\end{theorem}
\begin{remark}
The condition \eqref{ii} is in \cite[p.~179-180]{gia} called as
\emph{natural grow condition}. The result is new even for $\mathcal
B=\Gamma=0$ i.e. for q.c. solution to elliptic PDE with Lipschitz
coefficients.
\end{remark}
The proof of Theorem~\ref{mama} is given in Section~3. The methods
of the proof differ from the methods of the proof of corresponding
results for the class HQC. In Section~2 we make some estimates
concerning the Green function of the disk, and some estimates
concerning the gradient of a solution to elliptic partial
differential inequality, satisfying certain boundary condition
similar to those { in} the paper of Nagumo \cite{nagumo}. We first
prove interior estimates { for} the gradient of a solution $u$ of
elliptic PDE in terms of constants of the elliptic operator, and
modulus of continuity of $u$ (Theorem~\ref{fid}). After that we
recall a theorem of Nagumo (\cite{nagumo}), which shows that if $u$
is a solution of elliptic PDE, with vanishing boundary condition
defined in a domain $D$ whose boundary has { a} bounded curvature
from above by a constant $\kappa$, then $|\nabla u(z)|\le \gamma ,\
\ z\in D$, where $\gamma$ is a constant  not depending on $u$
providing that $64 \mathcal B \Gamma \|u\|_\infty <\pi$
(Theorem~\ref{aprio}). In order to prove Theorem~\ref{mama}, we
previously show that the function $u = |w|$ satisfies a certain
elliptic differential inequality near the boundary of the unit disk.
In order to show a priory bound, we make use of Mori's theorem which
implies that the modulus of continuity of a $K$-q.c. self-mapping of
the unit disk depends only on $K$. By using Theorem~\ref{fid}, we
show that the gradient is a priory bounded { on} compacts of the
unit disk, while Theorem~\ref{aprio} serves to obtain the a priory
bound of { the} gradient of $u$ in some "neighborhood" of the
boundary of the unit disk. By using the quasiconformality, we prove
that $\nabla w$ is a priory bounded as well.
\section{Auxiliary results}\subsection{Green function}
If $h(z,w)$ is a real function, then by $\nabla_z h$ we denote the
gradient $(h_x,h_y)$.
\begin{lemma}
{ If} $$h(z,w) = \log \frac{|1-z\bar{w}|}{|z-w|},$$ then
\begin{equation}\label{oneeq}\nabla_z h(z,w) = \frac{1-|w|^2}{(\bar z- \bar w)( w\bar
z-1)}\end{equation} and
\begin{equation}\label{twoeq}\partial_{w}\nabla_z h(z,w)=-\frac{1}{(1-w\bar z)^2},\
\  \partial_{\bar w}\nabla_zh(z,w)=-\frac{1}{(\bar w-\bar
z)^2}.\end{equation}
\end{lemma}
\begin{proof} First of all
$$\nabla_z h = (h_x,h_y)=h_x + ih_y.$$
Since
$$h_{\bar z} = \frac{1}{2}(h_x + i
h_y),$$ it follows that $$\nabla_z h = 2h_{\bar z}.$$ Since $$2h(z)
= \log \left(\frac{1-z\bar{w}}{z-w}\frac{1-\bar{z}{w}}{\bar{z}-\bar
w}\right),$$ {by differentiating} we obtain $$2h_{\bar z}(z)=
\log\left(\frac{1-\bar{z}{w}}{\bar{z}-\bar w}\right)_{\bar
z}=\frac{|w|^2-1}{(\bar z-\bar w)^2}\frac{\bar{z}-\bar
w}{{1-\bar{z}{w}}}.$$ This implies \eqref{oneeq}. From
$$\frac{1-|w|^2}{(\bar z- \bar w)( w\bar z-1)}=\frac{w}{w\bar
z-1}+\frac{1}{\bar w-\bar z}$$ it follows \eqref{twoeq}.
\end{proof}
\begin{corollary}  Let $G(\zeta, \omega)$ be the Green
function of the disk $\{\zeta:|\zeta-\zeta_0|\le R\}$ defined by
$$G(\zeta, \omega):=\log
\frac{|\varphi(\zeta)-\varphi(\omega)|}{|1-\varphi(\zeta)\overline{\varphi(\omega)}|},$$
where $$\varphi(\zeta) = \frac{1}{R}(\zeta-\zeta_0).$$ Then
\begin{equation}\label{first}|\nabla_{\zeta}G(\zeta,\omega)|\le \frac{2}{|\zeta-\omega|}\end{equation} and
\begin{equation}\label{sec}|\partial_{\omega_j}\nabla_{\zeta}G(\zeta,\omega)|\le
\frac{2}{|\zeta-\omega|^2}, j=1,2,\end{equation} where
$\omega=\omega_1+i\omega_2,$ $\omega_1,\omega_2\in\mathbf{R}$.
\end{corollary}
\begin{proof}
Let
$$\varphi(\zeta) = \frac{1}{R}(\zeta-z_0).$$ Then $$ \varphi'(\zeta) = \frac{1}{R}.$$
Take $z=\varphi(\zeta)$ and $w=\varphi(\omega)$ and define
$h(z,w)=G(\zeta,\omega)$. It follows that
\begin{equation}\label{compa}\nabla_{\zeta}G(\zeta,\omega)=  \nabla_z
h(z,w)\cdot \varphi'(\zeta)=\frac{1}{R}\nabla_z
h(z,w).\end{equation} Thus
\begin{equation}\label{comp}|\nabla_{\zeta}G(\zeta,\omega)|=\frac{1}{R} |\nabla_z
h(z,w)|.\end{equation} Further
\begin{equation}\label{rere}\frac{1-|w|^2}{|1-\bar z w|}\le \frac{1-|w|^2}{1-
|w|}\le 2.\end{equation}  Combining \eqref{rere}, \eqref{comp} {
with} \eqref{oneeq}, we obtain \eqref{first}. To get \eqref{sec},
observe first that for $\omega = \omega_1+i \omega_2$
\begin{equation}\label{xi}\partial_{\omega_1}=\partial_{\omega}+\partial_{\bar \omega}\end{equation} and
\begin{equation}\label{xibar}\partial_{\omega_2}=i(\partial_{\omega}-\partial_{\bar
\omega}).\end{equation} On the other hand, for $|z|\le 1$ and
$|w|\le 1$ we have
$$\left|\frac{1}{(1-w\bar z)^2}\right|\le
\left|\frac{1}{(w-z)^2}\right|.$$ From \eqref{xi}, \eqref{xibar},
\eqref{twoeq}, \eqref{compa} we deduce \eqref{sec}.
\end{proof}
\subsection{Interior estimates of gradient}
\begin{lemma}\label{lem}
Let $u:\overline{\mathbf U}\to\Bbb C$ be a continuous mapping. Then
there exists a positive function $\varpi=\varpi_u(t)$, $t\in (0,2)$,
such that $\lim_{t\to 0}\varpi_u(t)=0$ and $$|u(z)-u(w)|\le
\varpi(|z-w|), \ \ z,w\in \mathbf U.$$ The function $\varpi$ is
called the modulus of continuity of $u$
\end{lemma}
\begin{lemma}\label{meca}
Let $Y: D\to \mathbf U$ be a $C^2$ mapping of a domain $D\subset
\mathbf U$. Define $\mathbf
U(z_0,\rho):=\{z\in\mathbf{C}:|z-z_0|<\rho\}$ and assume that the
closure of $\mathbf U(z_0,\rho)$ is contained in $D$, and let $Z\in
\mathbf C$ be any complex number. Then we have the estimate:
\begin{equation}\label{2112}\begin{split}
|\nabla h(z_0)|&\le
\frac{2}{\rho^2}\int_{|y-z_0|=\rho}|Y(y)-Z|d\mathcal{H}^1(y)\
\end{split}
\end{equation}
where $h(z)$, $z\in\overline{\mathbf U(z_0,\rho)}$ is the Poisson
integral of $Y|_{z_0+\rho\mathbf T}$ and $\mathbf T$ is the unit
circle. Moreover $d\mathcal{H}^1$ is the Hausdorff probability
measure (i.e. normalized arc length measure).
\end{lemma}
\begin{proof}
Assume that $v\in C^2(\overline{\mathbf U})$ and define
\begin{equation}\label{vh}H(z)=\int_{\mathbf T}P(z,\eta)v(\eta)d\mathcal{H}^1(\eta),\end{equation} where
\begin{equation}\label{poisson}P(z,\eta)=\frac{1-|z|^2}{|z-\eta|^2}, \ \ \ |\eta|=1, \ \ |z|<1.\end{equation} Then $H$ is a harmonic function. It follows that
\begin{equation}\label{d3}\left<\nabla H(z),e\right>=\int_{\mathbf T}\left<\nabla_z P(z,\eta),e\right>
v(\eta)d\mathcal{H}^1(\eta),\ \ \ e\in \mathbf R^2.\end{equation} By
differentiating (\ref{poisson}), we obtain
$$\nabla_zP(z,\eta)=\frac{-2z}{|z-\eta|^2}-\frac{2(1-|z|^2)(z-\eta)}{|z-\eta|^{2+2}}.$$
Hence
$$\nabla_zP(0,\eta)=\frac{2\eta}{|\eta|^{4}}=2\eta.$$
Therefore
\begin{equation}\label{d1}|\left<\nabla_zP(0,\eta),e\right>|\le|\nabla_zP(0,\eta)\|e|=2|e|.\end{equation}
Using (\ref{d3}), (\ref{d1}), we obtain $$|\left<\nabla
H(0),e\right>|\le \int_{\mathbf T}|\nabla_zP(0,\eta)| |e|
|v(\eta)|d\mathcal{H}^1(\eta)=|e|\int_{\mathbf T}|\nabla_z
|v(\eta)|d\mathcal{H}^1(\eta).$$ Hence, we have
\begin{equation}\label{in}|\nabla H(0)|\le
2\int_{\mathbf T} |v(\eta)|d\mathcal{H}^1(\eta)\end{equation} Let
$v(z)=Y(z_0+\rho z)-Z$ and $H(z)=P[v|_{\mathbf T}](z)$. Then
$H(z)=h(z_0+\rho z)-Z$ and $\nabla H(0)=\rho \nabla h(z_0)$.
Inserting this into (\ref{in}), we obtain
\begin{equation}\label{in1}\begin{split}\rho|\nabla h(z_0)|&=
|\nabla H(0)|\le 2\int_{\mathbf T}
|Y(z_0+\rho\eta)-Z|d\mathcal{H}^1(\eta).\end{split}
\end{equation}
Introducing the change of variables $\zeta=z_0+\rho\eta$ in the
integral (\ref{in1}), we obtain
\begin{equation}\label{ini2}\begin{split} |\nabla h(z_0)|&\le
\frac{2}{\rho^2}\int_{|\zeta-z_0|=\rho}|Y(\zeta)-Z|d\mathcal{H}^1(\zeta)
\end{split}
\end{equation}
which is identical with (\ref{2112}).
\end{proof}
\begin{theorem}\label{fid}Let $D$ be a bounded domain, whose diameter is $d$.  Let $A(z)= \{a^{ij}(z)\}_{i,j=1}^2$ be a symmetric  matrix function
defined in a domain $\Omega\subset \mathbf C$ $(a^{ij}= a^{ji})$
satisfying the condition \eqref{(2.1)} and \eqref{con2}.
Let $u(z)$ be any $C^2$ solution of elliptic partial differential
inequality \eqref{(0)} such that
\begin{equation}\label{(3.3)} |u(z)| \le M \text{ in $D$.
}\end{equation} Then there exist constants $C^{(0)}$ and $C^{(1)}$,
depending on modulus of continuity of $u$, $\Lambda$, $\mathfrak L$,
$B$, $\Gamma$, $M$ and $d$ such that
\begin{equation}\label{drita}|\nabla u(z)| <
C^{(0)}\rho(z)^{-1}\max_{|\zeta-z|\le \rho(z)}\{|u(\zeta)|\} +
C^{(1)}\end{equation} where $\rho(z) = \mathrm{dist} (z, \partial
D)$.
\end{theorem}
\begin{proof} Fix a point $a\in D$ and
let $B_{p}$, $0<p< 1$, be a closed disk defined by
$$B_{p} = \{z; |z-a|\le {p}\ \mathrm{dist}(a,\partial D)\}.$$ Its
radius is
$$R_p={p}\ \mathrm{dist}(a,\partial D).$$
Define the function $\mu_p$ { as}
\begin{equation}\label{(3.4)} \mu_{p} = \max_{z\in B_{p}}\{|\nabla u|{r_p}(z)\}\end{equation}
where ${r_p}(z) = \mathrm{dist}\ (z,\partial B_{p})=R_p - |z-a|$.
Then there exists a point $z_p\in B_{p}$ such that
\begin{equation}\label{(3.5)}|\nabla
u(z_p)|{r_p}(z_p)=\mu_{p}    \ \ \  \        (z_p\in
B_{p}).\end{equation} We need the following result in the sequel.
\begin{lemma}\label{claim}
The function $\mu_{p}$ is continuous on $(0,1)$ and has { a}
continuous extension at $0$: $\mu_0=0$.
\end{lemma}
\begin{proof}[Proof of Lemma~\ref{claim}]
Let $p_n$ be a sequence converging to a number $p$, let
$\mu_{p_n}=|\nabla u(z_n)|{r_{p_n}}(z_n)$ and assume it converges to
$\mu'_p$. Prove that $\mu'_p= \mu_p$. Passing to a subsequence, we
can assume that $z_n \to z_p'$. Then $z_p'\in B_p$. Thus, $\mu'_p\le
\mu_p$. On the other hand, $\mu_{p_n}\ge |\nabla
u((1-\varepsilon_n)z_p)|{r_{p_n}}((1-\varepsilon_n)z_p)$, where
$\varepsilon_n$ is a positive sequence converging to zero. It
follows that $\mu'_p \ge \lim_{n\to \infty}|\nabla
u((1-\varepsilon_n)z_p)|{r_{p_n}}((1-\varepsilon_n)z_p)= \mu_p$.
Furthermore, since $r_p \le R_p = p\,\mathrm{dist}(a,\partial D)$,
we obtain $$\lim_{p\to 0^{+} }\mu_ p \le|\nabla u(0)|\lim_{p\to
0^{+}} R_p = 0.$$
\end{proof}
Now let $Tz = \zeta$ be a linear transformation of coordinates such
that
\begin{equation}\label{trans}\sum_{i,j=1}^2a^{ij}(z_p)D_{ij}u =\Delta
v,\end{equation}
where $v(\zeta) = u(z)$. 
By \cite[Lemma~11.2.1]{jost} the transformation $T$ can be chosen so
that \begin{equation}\label{T}T=\left(
                                                                 \begin{array}{cc}
                                                                   {\lambda_1}^{-\frac 12} & 0 \\
 0 & {\lambda_2}^{-\frac 12} \\
\end{array}
                                                     \right)\cdot
R,\end{equation} where $\lambda_1$ and $\lambda_2$ are eigenvalues
of the matrix $A(z_p)$ and $R$ is some orthogonal matrix. Then
$$\frac 1{\Lambda}\le \lambda_1,\lambda_2\le \Lambda.$$
 Let
$\nabla^2 u$ denotes the Hessian matrix of $u$:
$$\nabla^2 u=\left(
                                                                   \begin{array}{cc}
                                                                     D_{11} u & D_{12} u \\
    D_{21} u & D_{22} u \\
\end{array}\right).$$
Since
$$\nabla^2 u = T^t \nabla^2 v T,$$ we obtain:
\[\begin{split}
\mathrm{Trace}(A^t \nabla^2 u) &= \mathrm{Trace}(A^t T^t \nabla^2 v
T)\\&=\mathrm{Trace}((TA)^t \nabla^2 v T)\\&=\mathrm{Trace}(\nabla^2
v T (TA)^t )\\&= \mathrm{Trace}(\nabla^2 v T A^ t T^t)\\&=
\mathrm{Trace}( B^t \nabla^2 v ),\end{split}\] where
\begin{equation}\label{bb}B(\zeta) = T A(z) T^t.\end{equation}
Then
$$B(\zeta_p) = I,$$
%
\begin{equation}\label{uv} b^{ij}(\zeta)D_{ij} v(\zeta) \ =
a^{ij}(z)D_{ij}u(z),
\end{equation} where $B(\zeta)=\{b^{ij}\}_{i,j}=1^2$ and
\begin{equation}
\label{(3.6)} \Delta v= (\delta_{ij}-b^{ij}(\zeta))D_{ij}v +
b^{ij}(\zeta)D_{ij}v.\end{equation} Further,
$T(\mathbf{U}(z_p,r_p))\subset T(B_{p}) \subset T(D)=:D'$. From
\eqref{T} we see that $T(D(z_p,r_p))$  is an ellipse with axes equal
to $\lambda_1^{-1/2}\cdot r_p$ and $\lambda_2^{-1/2} \cdot r_p$ and
with the center at $\zeta_p=T(z_p)$. Then
 $D_\lambda:=\{\zeta: |\zeta-\zeta_p|\le \lambda r_p\}$ is a closed disk in $T(B_{p})$ provided that \begin{equation}\label{lala}
0 < \lambda< \frac{1}{2\sqrt{\Lambda}}.\end{equation} Let $G(\zeta,
\omega)$ be the Green function of the disk $D_\lambda$. So that,
from \eqref{(3.6)}
\[\begin{split} v &=-\frac{1}{\pi}\int_{D_\lambda} G(\zeta,\omega)(\delta_{ij}-b^{ij}(\omega))D_{ij} v(\omega)d\mathcal L^2(\omega)\\&\ \ \ \ \ -\frac{1}{\pi}\int_{D_\lambda}
G(\zeta,\omega)b^{ij}(\omega)D_{ij} v(\omega)d\mathcal L^2(\omega) +
h(\zeta),\end{split}\] where $d\mathcal L^2(z)=dxdy$ is the Lebesgue
two-dimensional measure in the complex plane and $h(\zeta)$ is the
harmonic function which takes the same { values} as $v(\zeta)$ for
$\zeta\in
\partial D_\lambda$. Then
\begin{equation}\label{(3.7)} |\nabla v(\zeta_p)|\le \mathcal{P} + \mathcal{Q} + \mathcal{R},\end{equation} where
\[\begin{split} &\mathcal{P} = |\frac{1}{\pi}\int_{D_\lambda} \nabla_{\zeta} G(\zeta_p,\omega)b^{ij}(\omega)D_{ij} v(\omega)d\mathcal L^2(\omega) |\\ &\mathcal{Q} =
|\frac{1}{\pi}\int_{D_\lambda} \nabla_{\zeta}
G(\zeta_p,\omega)(\delta_{ij}-b^{ij}(\omega))D_{ij}
v(\omega)d\mathcal L^2(\omega)|
\\&\mathcal{R} =|\nabla_{\zeta}h(\zeta_p)|.\end{split}\]
Further, it follows by  \eqref{con2} that $A$ is differentiable
almost everywhere. From  \eqref{bb} we obtain
$$D B(\zeta) \cdot T= T \cdot D A(z)\cdot T^{t},\ \ \ \text{for a.e.}\ \  z.$$ Here $D A(z)$ is the differential operator defined by
$$A(z+h)=A(z)+DA(z)h+o(|h|).$$ Notice that $DA(z)h$ is a matrix. Since
$\Lambda^{-1/2}|z|\le |T z|\le \Lambda^{1/2}|z|$, having in mind
\eqref{con2}, we obtain
\begin{equation}\label{B}\|D B(\zeta)|\le |T|^3\|D A(z)\|\le \Lambda^{3/2} \mathfrak L.\end{equation}
In the previous formula we mean the following norms: the norm of a
matrix $L$ is defined by  $|L|=\max\{|L h|: |h|=1\}$, and the norm
of an operator $DX(z)$ by $\|D X(z)\|=\max\{|D A(z)h|: |h|=1\}$,
($X=A,B$). Thus
\begin{eqnarray}\label{(3.8)}
   |B(\zeta)-B(\zeta_p)| &=& |B(\zeta)-\mathrm{I}|\le \Lambda^{3/2}\mathfrak L|\zeta -
   \zeta_p|
\end{eqnarray} As
$$|T(z)-T(z_p)|\le \lambda {r_p}(z_p),$$ by using the
inequalities
$$r_p(z_p)\le d(z,z_p)+{r_p}(z),$$
$$ d(z,z_p)\le \Lambda^{1/2} |T(z)-T(z_p)|$$
and by \eqref{(3.4)}, $$|\nabla u(z)|{r_p}(z)\le \mu_{p},$$ we
obtain
$$|\nabla u(z)|\le (1-\lambda
\Lambda^{1/2})^{-1}{r_{{p}}(z_p)}^{-1}\mu_{{p}} \text{ for $z \in
T^{-1} (D_\lambda)(\subset B_{p})$}. $$ From \eqref{lala} we obtain
that
\begin{equation}\label{lale}(1-\lambda
\Lambda^{1/2})^{-2}<4.\end{equation} Having in mind the formula,
$\nabla u(z) = \nabla v(\zeta)\cdot T$ we obtain
\begin{equation}
\label{thre} |\nabla v(\zeta)|\le 2 \Lambda^{1/2}
{r_p}(z_p)^{-1}\mu_{p}\end{equation}for $\zeta \in D_{\lambda}$.\\
Since $$|a^{ij}(z)D_{ij}u|\le\mathcal{B}|\nabla u|^2 + \Gamma,$$
$$|b^{ij}(\zeta)D_{ij} v(\zeta)| \ = |a^{ij}(z)D_{ij}u(z)|,$$ it follows that
\begin{equation}\label{tg}|b^{ij}(\zeta)D_{ij} v(\zeta)|\le \mathcal{B}|T|^2|\nabla
v|^2+\Gamma=\mathcal{B}\Lambda|\nabla v|^2+\Gamma
\end{equation} and therefore, from \eqref{thre}
{ we find that} \begin{equation} \label{(3.10)} |b^{ij}(\zeta)D_{ij}
v(\zeta)| \ \le 4\Lambda^2\mathcal{B}{r_p}(z_p)^{-2}\mu_{p}^2
+\Gamma.\end{equation} { Now  we} divide the proof into four steps:
\\
 {\it Step 1: Estimation of
$\mathcal{P}$}. From \eqref{first} and \eqref{(3.10)} we first have
\[\begin{split}
|\frac{1}{\pi}&\int_{D_\lambda} \nabla_{\zeta}
G(\zeta_p,\omega)b^{ij}(\omega)D_{ij} v(\omega)d\mathcal
L^2(\omega)|\\&\le \frac{2}{\pi}\int_{|\omega-\zeta_p|\le \lambda
{r_p}(z_p)} \frac{1}{|\omega-\zeta_p|}|b^{ij}(\omega)D_{ij}
v(\omega)|d\mathcal L^2(\omega)\\&\le
\frac{2}{\pi}\int_{|\omega-\zeta_p|\le \lambda {r_p}(z_p)}
\frac{1}{|\omega-\zeta_p|}(4\Lambda^2\mathcal{B}{r_p}(z_p)^{-2}\mu_{p}^2
+\Gamma)d\mathcal L^2(\omega) \end{split}\] Therefore
\begin{equation}\label{(3.11)} \mathcal{P} \le \frac{16\Lambda^2\mathcal{B}\lambda\mu_{p}^2}{{r_p}}
+4\Gamma {r_p}\lambda.
\end{equation}
{\it Step 2: Estimation of $\mathcal{Q}$.} Let $\mathbf{n_\omega}=
(\cos\alpha_1,\cos\alpha_2)$ be the unit inner vector of $\partial
D_\lambda$ at $\omega$. Then from { Green's} formula
$$\int_{\partial D_\lambda} \sum_{i=1}^2 u_i(\omega)\cos \alpha_i
d\mathcal{H}^1(\omega)= \int_{D_\lambda} (\partial_{\omega_1}u_1+
\partial_{\omega_2}u_2) d\mathcal L^2(\omega),$$ proceeding as in
\cite[Theorem~2]{nagumo}, we obtain
\begin{equation}\begin{split} \mathcal{Q} &\le |\frac{1}{\pi}\int_{|\omega-\zeta_p|= \lambda {r_p}(z_p)} \nabla_{\zeta} G(\zeta_p,\omega)
(\delta_{ij}-b^{ij}(\omega))\partial_i v(\omega)\cos\alpha_j
d\mathcal{H}^1(\omega)|\\&+ |\frac{1}{\pi}\int_{|\omega-\zeta_p|\le
\lambda {r_p}(z_p)} \nabla_{\zeta}
G(\zeta_p,\omega)\partial_{\omega_j}b^{ij}(\omega)\partial
_iv(\omega)d\mathcal L^2(\omega)|
\\&+|\frac{1}{\pi}\int_{|\omega-\zeta_p|\le \lambda {r_p}(z_p)} \partial_{\omega_j}\nabla_{\zeta} G(\zeta_p,\omega)(\delta_{ij}-b^{ij}(\omega))\partial_i v
(\omega)d\mathcal L^2(\omega)|.
\end{split}
\end{equation}
By using the Cauchy-Schwarz inequality, \eqref{first}, \eqref{sec},
\eqref{B}, \eqref{(3.8)}, \eqref{thre}, we obtain
 \begin{equation*}\mathcal{Q}\le 8\Lambda^2 \mathfrak L\lambda\mu_{p}+4\Lambda^2 \mathfrak L\lambda\mu_{p}+4\Lambda^2 \mathfrak L\lambda\mu_{p},\end{equation*} i.e.
\begin{equation}\label{(3.12)}\mathcal{Q}\le 16\Lambda^2 \mathfrak L\lambda\mu_{p}\end{equation}
{\it Step 3: Estimation of $\mathcal{R}$}.
\\
Let $\varpi(t)=\varpi_v(t)$ be the modulus of continuity of $v$ as
in Lemma~\ref{lem}. From \eqref{2112}, for $Z= v(\zeta_p)$ ($Z=0$),
$Y(\zeta) =v(\zeta)$ and $\rho=\lambda{r_p}(z_p)$, by using
Lemma~\ref{meca} and \ref{lem}, we obtain
\begin{equation}\label{(3.13)} \begin{split}\mathcal{R} \le |\nabla h(z_p)|&\le
\frac{2}{\lambda^2{r_p}(z_p)^2}\int_{|\omega-\zeta_p|=\lambda{r_p}(z_p)}|v(\omega)-Z|d\mathcal{H}^1(\omega)
\\&\le \frac{2}{\lambda{r_p}(z_p)}\max\{|v(\zeta)-Z|:
|\zeta-\zeta_p|=
\lambda {r_p}(z_p)\}\\
&\le\frac{\min\{2\varpi(\lambda
{r_p}(z_p)),2K\}}{\lambda{r_p}(z_p)},\end{split}\end{equation} where
\begin{equation}\label{K}K =
\sup_{|z-a|\le\rho(a)}|u(z)|.\end{equation} {\it Step 4: The finish
of the proof.} As
$$|\nabla v(\zeta_p)\ge \Lambda^{-1/2}|\nabla u(z_p)| =
\Lambda^{-1/2}{r_p}(z_p)^{-1}\mu_{p}$$ and ${r_p}(z_p)<2\rho(a)\le
d$,  from \eqref{(3.7)}, \eqref{(3.11)}, \eqref{(3.12)} and
\eqref{(3.13)}, { we get}
\begin{equation}\label{(3.14)}A_0\mu_{p}^2+ B_0\mu_{p} +C_0\ge 0,\end{equation}
where $$A_0 = {16\mathcal{B}\Lambda^2\lambda},$$
$$B_0 = 16\Lambda^2 \mathfrak L\lambda{{r_p}(z_p)}-\Lambda^{-1/2}$$ and
$$C_0
=4\Gamma {r_p}^2(z_p)\lambda+\frac{2\min\{\varpi(\lambda
{r_p}(z_p)),K\}}{\lambda}.$$ We can take $\lambda>0$ depending on
$\varpi$, $\Lambda$, $\mathfrak L$, $B$, $\Gamma$ and $d$ so small
that
\begin{equation}\label{(3.15)} B_0^2> 4A_0C_0\end{equation} and
\begin{equation}\label{lam}
16\Lambda^2 \mathfrak L\lambda{{r_p}(z_p)}\lambda\le
1/2\Lambda^{-1/2}.
\end{equation}
 Let $\mu_1$ and
$\mu_2$ ($\mu_1<\mu_2$) be the distinct real roots of the equation
\begin{equation}\label{(3.16)} A_0\mu^2 + B_0 \mu + C_0= 0.\end{equation} Then
 from \eqref{(3.14)} { we have}
$$\mu_p \le \mu_1 \text{ or } \mu _{p} \ge \mu_2.$$
Lemma~\ref{claim} asserts that $\mu_{p}$ depends on ${p}$
continuously for $0 <{p} < 1$ and $\lim_{{p}\to 0}\mu_{p} = 0$. Then
we have only $\mu_{p}\le \mu_1$. And, letting ${p}$ tend to $1$, by
the definition of $\mu _{p}$
\begin{equation}\label{(3.17)} |\nabla u(a)|\le \mu_1\rho(a)^{-1}.\end{equation} As $\mu_1$ is the smaller root of \eqref{(3.16)},
\[\begin{split}\mu_1&=\frac{-B_0-\sqrt{B_0^2-4A_0C_0}}{2 A_0}\\&=
\frac{2C_0}{-B_0+\sqrt{B_0^2-4A_0C_0}}\\&\le
-\frac{2C_0}{B_0}.\end{split}\]  From \eqref{(3.17)} and \eqref{K}
we get
\begin{equation}\label{nna}|\nabla u(a)|\le
C^{(0)}\rho(a)^{-1}\sup_{|z-a|\le\rho(a)}|u(z)|+C^{(1)}\end{equation}
where $C^{(0)}$ and $C^{(1)}$ depend on  $\Lambda$, $\mathfrak L$,
$B$, $M$, $\Gamma$, $d$ and on modulus of continuity of $u$.
%
\end{proof}
%
%
%
%
%
\subsection{Boundedness of gradient}
\begin{definition}\label{remi} We say that a domain $D$ satisfies
the {\it exterior sphere condition} for some $\kappa>0$ {  if to}
any point $p$ of $\partial D$ there corresponds a ball $B_p\subset
\mathbf C$ with radius $\kappa$ such that $\overline D\cap
B_p=\{p\}$.
\end{definition}
\begin{theorem}[A priory bound] \label{aprio} \cite[Lemma~2]{nagumo}
 Let $D$ be a complex domain with diameter $d$ satisfying exterior sphere condition for some $\kappa>0$.
 Let $u(z)$ be a twice differentiable mapping satisfying
the elliptic differential inequality \eqref{(0)} in $D$ satisfying
the boundary condition $u = 0$ $(z \in G)$. Assume in addition that
$|u(z)|\le M$, $z\in D$, \begin{equation}\label{nagumo}
\frac{4}{\pi}\cdot 16 \mathcal{B}\Gamma M <1\end{equation} and $u\in
C(\overline D)$. Then
\begin{equation}\label{ego}|\nabla u |\le \gamma,\ \ z\in D,\end{equation} where $\gamma$
is a constant depending {  only} on $\kappa$, $M$, $\mathcal{B}$,
$\Gamma$, $\mathfrak L$, $\Lambda$ and $d$.
\end{theorem}
\begin{remark}
See \cite[Theorem~15.9]{gt} for a related result. In the statement
of \cite[Lemma~2]{nagumo} instead of condition \eqref{nagumo}
appears $$16 \mathcal{B}\Gamma M <1$$ { However,  a related
 proof} lays on
\cite[Theorem~2]{nagumo}, { wich}, it seems that works only under
the condition \eqref{nagumo}. Indeed, the right hand side of the
inequality in the first line on \cite[p.~214]{nagumo} should be
multiplied by $$\frac{2 \Gamma(1 + m/2)}{\sqrt\pi \Gamma((m +
1)/2)},$$ where $m$ is the dimension of the space (in our case
$m=2$) and
$$\frac{2 \Gamma(1 + 2/2)}{\sqrt\pi \Gamma((2 +
1)/2)}=\frac{4}{\pi}.$$
\end{remark}
\section{Proof of the main theorem}
We need the following lemmas.
\begin{lemma}\label{forhelp}\cite{jmaa}
Every $K-$q.r. mapping $w(z)=\rho(z) S(z):D\to \Omega$, $D,$
$\Omega,\subset \Bbb C$, $\rho=|w|$, $S(z)=e^{is(z)}, s(z)\in
[0,2\pi)$, satisfies the inequalities
\begin{equation}\label{edtreta}\rho|\nabla S|\le K|\nabla
\rho|\end{equation} and
\begin{equation}\label{ekaterta}|\nabla \rho|\le K\rho|\nabla S|
\end{equation}
almost everywhere on $D$. Inequalities \eqref{edtreta} and
\eqref{ekaterta} are sharp{ ;} the equality
\begin{equation}\label{edtreta3}\rho|\nabla S|=|\nabla
\rho|\end{equation} holds if $w$ is a $1$-quasiregular mapping. We
also have
\begin{equation}\label{nabur}K^{-1}|\nabla w|\le |\nabla \rho|\le |\nabla w|.\end{equation}
\end{lemma}
\begin{lemma}\label{meli}
If $w = \rho S:\mathbf U \to \mathbf U$, $\rho = |w|$, is twice
differentiable, then
\begin{equation}\label{(00)}L[\rho]=\rho(a^{11}|p|^2+2a^{12}\left<p,q\right>+a^{22}|q|^2)+\left<L[w], S\right>,\end{equation} where $p= D_1 S$ and $q=D_2 S$.

If in addition $w$ is $K-q.c.$ and satisfies
\begin{equation}\label{(000)}|L[w]|=|\sum_{i,j=1}^2 a^{ij}(z)D_{ij}w| \le \mathcal{B}|\nabla w|^2+\Gamma
,\end{equation} then there exists a constant $\Theta$ depending on
$K$, $\mathcal B$ and $\Gamma$ such that
\begin{equation}\label{lll}|L[\rho]|\le \frac{\Theta}{\rho}|\nabla \rho|^2
+\Gamma.\end{equation}
\end{lemma}
\begin{proof}
Let $w=(w_1,w_2)$ (here $w_i$ are real), $S=(S_1,S_2)$ and let
$f=(f_1,f_2)$. For real differentiable functions $a$ and $b$ define
the bi-linear operator
$$D[a,b] = \sum_{k,l=1}^2 a^{kl}(z)D_ka(z)D_l b(z).$$ Since $w_i=\rho
S_i$, $ i\in \{1,2\}$ and
$$\rho=\sum_{i=1}^2S_iw_i,$$ we obtain
\begin{equation}\label{sec1}
L[w_i]=S_iL[\rho] +\rho L[S_i] +2D[\rho, S_i], \ i \in \{1,2\}
\end{equation}
and
\begin{equation}\label{thir}
L[\rho]=\sum_{i=1}^2 w_iL[S_i]+\sum_{i=1}^2
S_iL[w_i]+2\sum_{i=1}^2D[S_i,w_i].
\end{equation} From (\ref{sec1}) we obtain
\begin{equation}\begin{split}\label{for}
L[\rho]&=L[\rho]|S|^2\\&= \sum_{i=1}^2S_i\cdot S_i L[\rho]
\\&=\sum_{i=1}^2S_iL[w_i]-\rho \sum_{i=1}^2 S_i L[S_i]-2
\sum_{i=1}^2S_iD[\rho,S_i].\end{split}
\end{equation}
By adding (\ref{thir}) and (\ref{for}) we obtain $$L[\rho]
=\sum_{i=1}^2(D[S_i,w_i]-S_i D[\rho,S_i]) +\left<L[w], S\right>.$$
On the other hand
\[\begin{split}D[S_i,w_i]-S_i D[S_i,\rho]&= \sum_{k,l=1}^2 a^{kl}(z)D_kS_i D_l w_i
-S_i \sum_{k,l=1}^2 a^{kl}(z)D_kS_i D_l \rho\\&=\sum_{k,l=1}^2
a^{kl}(z)D_kS_i (\rho D_l S_i+S_i D_l \rho) -S_i \sum_{k,l=1}^2
a^{kl}(z)D_kS_i D_l \rho\\&=\rho \sum_{k,l=1}^2 a^{kl}(z)D_kS_i D_l
S_i, \ \ i=1,2.\end{split}\] Thus \[\begin{split}L[\rho]&= \rho
\sum_{i,k,l=1}^2 a^{kl}(z)D_kS_i D_l S_i +\left<L[w],
S\right>\\&=\rho(a^{11}|p|^2+2a^{12}\left<p,q\right>+a^{22}|q|^2)+\left<L[w],
S\right>,
\end{split}\] where $p=(D_1S_1,D_1 S_2)$ and $q=(D_2 S_1, D_2 S_2)$.
Therefore \[\begin{split}|L[\rho]|&\le \Lambda \rho
(|p|^2+|q|^2)+(\mathcal B|\nabla w|^2+\Gamma)\\&=\Lambda \rho
\|\nabla S\|^2+(\mathcal B|\nabla w|^2+\Gamma),\end{split}\]
provided \eqref{(000)} holds. Here $\|\cdot\|$ is { the}
Hilbert-Schmidt norm which satisfies the inequality $\|P\|\le \sqrt
2 |P|$. If $w$ is $K-$q.c., then according to \eqref{edtreta} and
\eqref{edtreta3} we have
$$|L[\rho]|\le {2K\Lambda }|\nabla \rho|^2\rho^{-1}+(\mathcal B K|\nabla
\rho|^2+\Gamma).$$ Taking $\Theta =  2K \Lambda  + \mathcal B K$ we
obtain \eqref{lll}.
\end{proof}
\begin{lemma}\label{lemah}
If $f=u+iv$ is a $K$ q.c. mapping satisfying elliptic differential
inequality, then $u$ and $v$ satisfy the elliptic differential
inequality.
\end{lemma}
\begin{proof}
Let $$A:=|\nabla u|^2=2(|u_z|^2+|u_{\bar z}|^2)=\frac
12(|f_z+\overline{f_{\bar z}} |^2+|f_{\bar z}+\overline{f_z}|^2)$$
and $$B:=|\nabla v|^2=2(|v_z|^2+|v_{\bar z}|^2)=\frac
12(|f_z-\overline{f_{\bar z}} |^2+|f_{\bar z}-\overline{f_z}|^2).$$
Then
$$\frac{A}{B}=\frac{|1+\mu|^2}{|1-\mu|^2}$$ where
$\mu={\overline{f_{\bar z}}}/{f_z}$. Since $|\mu|\le
k=\frac{K-1}{K+1}$
\begin{equation}\label{help} \frac{(1-k)^2}{(1+k)^2}\le \frac AB
\le \frac{(1+k)^2}{(1-k)^2}.
\end{equation}
As $$|L[f]|=|L[u]+iL[v]|\le \mathcal B|\nabla f|^2+\Gamma\le
\mathcal B(|\nabla u|^2+|\nabla v|^2)+\Gamma,$$  the relation
(\ref{help}) yields
$$|L[u]|\le \mathcal B\left(1+\frac{(1+k)^2}{(1-k)^2}\right)|\nabla u|^2+\Gamma$$ and
$$|L[v]|\le \mathcal B\left(1+\frac{(1+k)^2}{(1-k)^2}\right)|\nabla v|^2+\Gamma.$$
\end{proof}
Before proving the main { results} of this paper let us recall one
of the most fundamental results concerning quasiconformal mappings.
\begin{proposition}[Mori]
If $w:\mathbf U\to \mathbf U$, $w(0)=0$, is a $K$ quasiconformal
harmonic mapping of the unit disk onto itself, then
$$|w(z_1)-w(z_2)|\le 16|z_1-z_2|^{1/K},\ \ z_1,z_2\in \mathbf U.$$
\end{proposition}
Mori's theorem for q.c. selfmappings of the unit disk has been
generalized in various directions in the plane and in the space. See
for example, the papers \cite{kotj}, \cite{gema} and \cite{FV}.

%
%
\begin{proof}[Proof of Theorem~\ref{mama}] The idea of the proof is to estimate the gradient of $w$ in some "neighborhood" of the boundary together with some interior estimate
in the rest of the unit disk.  { Put} $\alpha$, $\beta\in\mathbf{R}$
such that $\frac{1+|a|}{2}\le \alpha <1$ and $\beta =
\frac{\alpha+1}{2}$. Define $D_\alpha  = \{z:|z|\le \beta\}$ and
$A_\alpha=\{z:\alpha\le |z|<1\}$.

Let $w= (w_1,w_2)$. According to Theorem~\ref{fid} and
Lemma~\ref{lemah}, there { exist} a constant $C_i$ depending only on
modulus of continuity of $w_i$, $\mathcal B$, $\Gamma$, $K$,
$\Lambda$, $\mathfrak L$ and $\alpha$ such that
\begin{equation}\label{ww}|\nabla w_i(z)|\le C_i, \ \ z\in D_\alpha, i=1,2.\end{equation} By Mori's theorem, the modulus
of continuity of $w_i$ depends only on $K$ and $a$. Thus
\begin{equation}\label{mat}|\nabla w(z)|\le |\nabla w_1|+|\nabla
w_2|\le C_1+C_2=C_3(K,\mathcal B,\Gamma,\Lambda, \mathfrak L,
\alpha),\ \ \ z\in D_\alpha.\end{equation} As $w$ is $K$
quasiconformal selfmapping of the unit disk, by Mori's theorem
(\cite{wang}) it satisfies the inequality:
\begin{equation}\label{mori}
{4^{1-K}}\left|\frac{a-z}{1-z\bar a}\right|^K \le |w(z)|,\ \ |z|<1,
\end{equation}
where $a=w^{-1}(0)$. Let $u = |w|$. From Lemma~\ref{meli} and
\eqref{mori} { we find that}
\begin{equation}\label{afte}|L[u]|\le
 2^{3K-2}\left(\frac{1+|a|}{1-|a|}\right)^K{\Theta}|\nabla u|^2 +\Gamma,
(1+|a|)/2<|z|<1.
\end{equation}
Let $g$ be a function
$$g:A_\alpha\to \mathbf R$$ defined { as}
$$g(z)=\left\{
        \begin{array}{ll}
          1, & \hbox{if $\beta<|z|\le 1$;} \\
          1+(u(z)-1)\frac{\exp\frac{1}{|z|^2-\beta^2}}{\exp\frac{1}{\alpha^2-\beta^2}}, & \hbox{if $\alpha\le|z|\le \beta$.}
        \end{array}
      \right.$$
{ Define} $$\phi(z):
=\frac{\exp\frac{1}{|z|^2-\beta^2}}{\exp\frac{1}{\alpha^2-\beta^2}}.$$
Then
$$L[g]=\left\{
         \begin{array}{ll}
           0, & \hbox{if $\beta<|z|\le 1$;} \\
           (u(z)-1)L[\phi]+\phi L[u]+D[u,\phi], & \hbox{if $\alpha\le|z|\le \beta$.}
         \end{array}
       \right.
$$
Therefore \begin{equation}\label{aka}|L[g]|\le\left\{
         \begin{array}{ll}
           0, & \hbox{if $\beta<|z|\le 1$;} \\
           \mathcal B_1|\nabla u|^2+\Gamma_1, & \hbox{if $\alpha\le|z|\le \beta$,}
         \end{array}
       \right.\end{equation} where $$\mathcal
B_1=2^{3K-2}\left(\frac{1+|a|}{1-|a|}\right)^K\left( 2K \Lambda +
{\mathcal B K}\right)$$ and $\Gamma_1$ is a constant depending only
on $K,$ $\mathcal B,$ $\Gamma$, $\Lambda$, $\mathfrak L$ and
$\alpha$. By \eqref{nabur}, \eqref{mat} and \eqref{aka} we have
\begin{equation}\label{oka}|L[g]|\le C_4(K, \mathcal B, \Gamma,
\Lambda, \mathfrak L, \alpha), \ \ z\in A_\alpha\end{equation} and
\begin{equation}\label{oko}|\nabla g|\le C_5(K, \mathcal B, \Gamma, \Lambda,\mathfrak L, \alpha), \ \ z\in A_\alpha.\end{equation}
Furthermore, by using the inequalities \eqref{afte}, \eqref{oka},
\eqref{oko} and $|a+b|^2\le 2(|a|^2+|b|^2)$, we have
\[\begin{split}|L[u-g]|&\le |L[u]|+|L[g]|\\&\le \mathcal B_1|\nabla
u|^2+C_7(K, \mathcal B, \Gamma, \Lambda, \mathfrak L, \alpha)\\&\le
2\mathcal B_1 |\nabla u-\nabla g|^2+C_8(K, \mathcal B, \Gamma,
\Lambda,\mathfrak L, \alpha),\ \ z\in A_\alpha.\end{split}\]  By
Mori's theorem, there exists a constant $\alpha=\alpha(K,a)<1$ such
that
$$M = \max\{|u(z)-g(z)|: z\in A_\alpha\}$$ is small enough, satisfying the inequality
\begin{equation}\label{bana}\frac{64}{\pi}\cdot 2\mathcal B_1 M \Lambda<1.
\end{equation}
Thus $\tilde u = u-g$ satisfies the conditions of
Theorem~\ref{aprio} in the domain $D=A_\alpha$. The conclusion is
that $\nabla u$ is bounded in $\beta<|z|<1$ by a constant depending
only on $K$, $\mathcal B$, $\Gamma$, $\Lambda$, $\mathfrak L$ and
$a$ and on { the} modulus of continuity of $\tilde u$. From Mori's
theorem, the modulus of continuity of $u$ depends only on $K$ and
$a$. Combining \eqref{oko} with \eqref{nabur}, we obtain
\begin{equation}\label{vuo}|\nabla w|\le C_0(K,\mathcal B, \Gamma,\Lambda, \mathfrak L, a),\
\ \beta<|z|<1.\end{equation} From \eqref{mat} and \eqref{vuo} we
obtain the desired conclusion.
\end{proof}
\subsection*{Acknowledgment} I am thankful to the referee for providing constructive comments and
help in improving the contents of this paper.
\bibliographystyle{amsplain}
\providecommand{\bysame}{\leavevmode\hbox
to3em{\hrulefill}\thinspace}
\providecommand{\MR}{\relax\ifhmode\unskip\space\fi MR }
\providecommand{\MRhref}[2]{%
  \href{http://www.ams.org/mathscinet-getitem?mr=#1}{#2}
} \providecommand{\href}[2]{#2}

\end{document}